\documentclass[a4paper,11pt]{article}
\usepackage{amsmath, amsfonts, amscd, amssymb, amsthm, enumerate, xypic}

\def\bfB{\mathbf{B}}

\DeclareMathOperator{\id}{\operatorname{id}}

\DeclareMathOperator{\Mat}{\operatorname{M}}
\DeclareMathOperator{\Matt}{\operatorname{T}}
\DeclareMathOperator{\Hom}{\operatorname{Hom}}

\DeclareMathOperator{\Ker}{\operatorname{Ker}}

\DeclareMathOperator{\Mats}{\operatorname{S}}
\DeclareMathOperator{\MatT}{\operatorname{T}}

\DeclareMathOperator{\End}{\operatorname{End}}

\DeclareMathOperator{\GL}{\operatorname{GL}}
\DeclareMathOperator{\Vect}{\operatorname{span}}
\DeclareMathOperator{\im}{\operatorname{Im}}
\DeclareMathOperator{\tr}{\operatorname{tr}}

\DeclareMathOperator{\car}{\operatorname{char}}

\renewcommand{\setminus}{\smallsetminus}
\renewcommand{\epsilon}{\varepsilon}

\def\F{\mathbb{F}}
\def\R{\mathbb{R}}
\def\C{\mathbb{C}}

\def\calF{\mathcal{F}}

\def\calH{\mathcal{H}}

\def\calM{\mathcal{M}}

\def\calS{\mathcal{S}}
\def\calT{\mathcal{T}}

\def\calW{\mathcal{W}}
\def\calX{\mathcal{X}}
\def\calY{\mathcal{Y}}
\def\calZ{\mathcal{Z}}

\def\lcro{\mathopen{[\![}}
\def\rcro{\mathclose{]\!]}}

\theoremstyle{definition}
\newtheorem{Def}{Definition}[section]
\newtheorem{Not}[Def]{Notation}

\theoremstyle{plain}
\newtheorem{theo}{Theorem}[section]
\newtheorem{prop}[theo]{Proposition}

\newtheorem{lemma}[theo]{Lemma}
\newtheorem{claim}{Claim}[section]

\theoremstyle{plain}

\theoremstyle{remark}
\newtheorem{Rems}{Remarks}[section]
\newtheorem{Rem}[Rems]{Remark}

\title{Spaces of triangularizable matrices (II): Finite fields with odd characteristic}
\author{Cl\'ement de Seguins Pazzis\footnote{Universit\'e de Versailles Saint-Quentin-en-Yvelines, Laboratoire de Math\'ematiques
de Versailles, 45 avenue des Etats-Unis, 78035 Versailles cedex, France}
\footnote{e-mail address: clement.de-seguins-pazzis@ac-versailles.fr}}

\begin{document}

\thispagestyle{plain}

\maketitle
\begin{abstract}
Let $\F$ be a field. Denote by $t_n(\F)$ the greatest possible dimension for a vector space of $n$-by-$n$ matrices over $\F$
in which every element is triangularizable over $\F$.

It was recently proved that $t_n(\F)=\frac{n(n+1)}{2}$ if and only if $\F$ is not quadratically closed.
The structure of the spaces of maximal dimension was also elucidated provided $\F$ is infinite and not quadratically closed.
In this sequel, we extend this result to finite fields with odd characteristic. More specifically, we prove that if $\F$ is finite with odd characteristic, then the space of all upper-triangular $n$-by-$n$ matrices
is, up to conjugation, the sole vector space of $n$-by-$n$ matrices that has dimension $\frac{n(n+1)}{2}$
and consists only of triangularizable matrices.
\end{abstract}

\vskip 2mm
\noindent
\emph{AMS MSC:} 15A30, 15A03, 15A18

\vskip 2mm
\noindent
\emph{Keywords:} triangularization, spectrum, spaces of matrices, dimension, finite fields


\section{Introduction}

\subsection{The problem, and the main result}

Let $\F$ be a field. We denote by $\Mat_n(\F)$ the algebra of all $n$-by-$n$ matrices with entries in $\F$, by $\Matt_n(\F)$
its subalgebra of all upper-triangular matrices, and by $\mathfrak{sl}_n(\F)$ the linear subspace of all trace zero matrices.

A subset $\calX$ of $\Mat_n(\F)$ is called \textbf{weakly triangularizable} when all its elements are triangularizable over $\F$,
i.e., conjugated to an element of $\Matt_n(\F)$, or equivalently when each element of $\calX$ is annihilated by a split polynomial with coefficients in $\F$.
We adopt a similar definition for subsets of $\End(V)$ when $V$ is a finite-dimensional vector space over $\F$.

The study of weakly triangularizable subspaces can be viewed as a variation on the one of linear subspaces of nilpotent matrices, initiated by Murray Gerstenhaber \cite{Gerstenhaber},
where nilpotence can be seen as short for ``similar to a strictly upper-triangular matrix".
Strikingly, meaningful results on weakly triangularizable subspaces have emerged only very recently, a plausible reason being that
the solution intimately depends on the arithmetic properties of
the underlying field. For example, every subset of $\Mat_n(\C)$ is weakly triangularizable, but in $\Mat_n(\R)$ the maximal possible dimension for a weakly triangularizable
subspace is $\frac{n(n+1)}{2}$, which is precisely $\dim \Matt_n(\R)$:
for an efficient proof of the latter fact, simply note that among the skew-symmetric real matrices the only one that is triangularizable over the reals is the zero matrix
(since all these matrices have pure imaginary eigenvalues), and that the dimension of the space of all $n$-by-$n$ skew-symmetric real matrices is $\frac{n(n-1)}{2}\cdot$

Naturally, the questions that are raised for weakly triangularizable subspaces are the analogues of the ones that are raised in the study of linear or affine subspaces of matrices
with a distinctive individual property:
\begin{enumerate}[(a)]
\item What is the greatest possible dimension $t_n(\F)$ for such a linear subspace?
\item Can one give a constructive description of the weakly triangularizable subspaces with dimension $t_n(\F)$, which we call the
\textbf{optimal} weakly triangularizable subspaces?
\end{enumerate}

In any case, it is clear that $t_n(\F)$ is greater than or equal to $\frac{n(n+1)}{2}$,
the dimension of $\Matt_n(\F)$, and it turns out that, excluding finite fields with characteristic $2$, there is a
simple characterization of the fields for which $t_n(\F)=\frac{n(n+1)}{2}$:

\begin{theo}[Theorem 1.1 in \cite{dSPtriang}]\label{theo:dimNRC}
Let $n \geq 2$. Assume that $\F$ is infinite or $\car(\F) \neq 2$.
Then the following conditions are equivalent:
\begin{enumerate}[(i)]
\item The greatest possible dimension for a weakly triangularizable linear subspace of $\Mat_n(\F)$ is $\frac{n(n+1)}{2}\cdot$
\item The field $\F$ is not quadratically closed.
\end{enumerate}
\end{theo}

A close examination of the proof of this result in \cite{dSPtriang} reveals that it also holds for
finite fields with characteristic $2$ provided that their cardinality is greater than or equal to $n$,
and we conjecture that the result actually holds for all fields.

To avoid any confusion, let us be clear that to us a field $\F$ is quadratically closed when it has no algebraic extension of degree $2$,
which is equivalent to the surjectivity of $x \in \F \mapsto x^2 \in \F$ if $\car(\F) \neq 2$, but not if $\car(\F) = 2$.
Hence, according to this definition no finite field, even of characteristic $2$, is quadratically closed.

From now on, we only consider fields that are not quadratically closed.
For infinite fields (and, once again, a close examination of the proofs of \cite{dSPtriang} reveals that the milder assumption $|\F| \geq n$ is sufficient for the proofs to
hold), the optimal weakly triangularizable subspaces have been essentially determined.
To state the results, we need to recall two notations:
\begin{itemize}
\item The one of similarity between matrix spaces: two subsets $\calX$ and $\calY$ of $\Mat_n(\F)$ are called similar, and we write $\calX \simeq \calY$, when
$\calX=P \calY P^{-1}$ for some $P \in \GL_n(\F)$ ; every subset that is similar to a weakly triangularizable subspace is weakly triangularizable.
\item The one of the joint of subsets $\calX_1 \subseteq \Mat_{n_1}(\F),\dots, \calX_p \subseteq \Mat_{n_p}(\F)$ of square matrices (potentially with different sizes),
denoted by $\calX_1 \vee \cdots \vee \calX_p$ and consisting of all the block upper-triangular matrices of the form
$$\begin{bmatrix}
X_1 & & [?] \\
& \ddots &  \\
[0] & & X_p
\end{bmatrix}$$
where $X_1 \in \calX_1,\dots,X_p \in \calX_p$ and the blocks above the diagonal are entirely arbitrary.
Note that if $\calX_1,\dots,\calX_p$ are weakly triangularizable subspaces, so is their joint.
\end{itemize}

Denote by $\Mats_n(\F)$ the space of all $n$-by-$n$ symmetric matrices with entries in $\F$.
The following result combines theorems 1.2 and 1.3 from \cite{dSPtriang}:

\begin{theo}\label{theo:classNRC}
Let $\F$ be an infinite field such that $x \in \F \mapsto x^2 \in \F$ is nonsurjective.
Let $\calS$ be a weakly triangularizable linear subspace of $\Mat_n(\F)$ with dimension $\frac{n(n+1)}{2}\cdot$
Then there exists a (unique) partition $n=n_1+\cdots+n_p$ such that $\calS \simeq \Mats_{n_1}(\F) \vee \cdots \vee \Mats_{n_p}(\F)$.
\end{theo}

This theorem however does not give the full description of the optimal weakly triangularizable subspaces,
which requires that we have access to the integers $k \geq 2$ for which every matrix of $\Mats_k(\F)$
is triangularizable. It is clear that if $\Mats_k(\F)$ is weakly triangularizable, then so is $\Mats_i(\F)$ for all $i \in \lcro 2,k\rcro$.
Moreover, it is known (this is an easy exercise) that when $\car(\F) \neq 2$, the space
$\Mats_2(\F)$ is weakly triangularizable if and only if $\F$ is Pythagorean (i.e., in $\F$ the sum of any two squares is a square).
In particular, if $\F$ is infinite, non-Pythagorean, and with characteristic other than $2$,
then the conclusion of Theorem \ref{theo:classNRC} is that $\Matt_n(\F)$ is, up to conjugation, the sole
 weakly triangularizable subspace of $\Mat_n(\F)$ with dimension $\frac{n(n+1)}{2}\cdot$

In contrast, all matrix spaces $\Mats_k(\R)$ are weakly triangularizable,
and hence there are as many conjugacy classes of optimal weakly triangularizable subspaces of $\Mat_n(\R)$ as there are ordered partitions of the integer $n$.

For perfect fields with characteristic $2$, the result is substantially different but, in a way, simpler.
The following result combines theorems 1.2 and 1.5 from \cite{dSPtriang}:

\begin{theo}\label{theo:classperfect}
Let $\F$ be a perfect infinite field with $\car(\F) = 2$. Assume that $\F$ is not quadratically closed.
Let $\calS$ be a weakly triangularizable linear subspace of $\Mat_n(\F)$ with dimension $\frac{n(n+1)}{2}\cdot$
Then $\calS$ is conjugated to a joint $\calS_1 \vee \cdots \vee \calS_p$, where each $\calS_i$ equals either $\Mat_1(\F)$ or $\mathfrak{sl}_2(\F)$.
\end{theo}

In contrast with Theorem \ref{theo:classNRC}, it is very easy to see that each space $\Mat_1(\F)$ and $\mathfrak{sl}_2(\F)$
is weakly triangularizable if $\F$ is perfect with characteristic $2$, and hence any joint of such spaces is weakly triangularizable.

The aim of the present work is to extend Theorem \ref{theo:classNRC} to any finite field with odd characteristic.
In that case the statement of the result is even simpler since no finite field of odd characteristic is Pythagorean, and hence we are simply aiming at proving
the following result:

\begin{theo}[Main theorem]\label{theo:main}
Let $\F$ be a finite field with $\car(\F) \neq 2$.
Then every weakly triangularizable linear subspace of $\Mat_n(\F)$ with dimension $\frac{n(n+1)}{2}$ is similar to $\Matt_n(\F)$.
\end{theo}

Finite fields with characteristic $2$ behave completely differently, nevertheless we hope to settle their case in the future.

\subsection{Strategy}

In \cite{dSPtriang}, Theorems \ref{theo:dimNRC} and \ref{theo:classNRC} were proved thanks to a double-duality method, connecting the problem
to the one of spaces of bounded rank matrices. The method used in a critical way a powerful theorem of Atkinson \cite{AtkinsonPrim},
which describes the so-called (semi-) primitive
spaces of bounded rank matrices. However, this method fails for fields with low cardinality, due to the failure of Atkinson's theorem
for such fields (see, e.g., \cite{dSPprimitiveF2} for various counterexamples over $\F_2$).

We must turn here to more elementary methods, and here we will take advantage of the major key in the proof of Theorem \ref{theo:dimNRC},
and combine it with the now classical method that was used in various works of the same flavor
\cite{dSPsoleeigenvalue,dSPlargeaffinenonsingular,dSPGerstenhaberskew,dSPfeweigenvalues}.
Let us immediately recall this result, in terms of spaces of endomorphisms.

\begin{Def}
Let $\calS$ be a linear subspace of $\End(V)$, where $V$ is a finite-dimensional vector space. A vector $x \in V \setminus \{0\}$
is called \textbf{$\calS$-adapted} when $\calS$ contains no element with range $\F x$ and trace $0$.

A linear hyperplane $H$ of $V$ is called \textbf{$\calS$-adapted} when $\calS$ contains no element with kernel $H$ and trace $0$.
\end{Def}

\begin{prop}[Proposition 2.1 in \cite{dSPtriang}]\label{prop:adapted}
Assume that $x \in \F \mapsto x^2 \in \F$ is nonsurjective. Let $V$ be a vector space over $\F$ with positive finite dimension.
Then every weakly triangularizable linear subspace $\calS$ of $\End(V)$ has an adapted vector.
\end{prop}

This result will allow us to find an inductive proof of Theorem \ref{theo:main}.
Before we sketch the proof, we restate Theorem \ref{theo:main} in geometrical terms.

\begin{Not}
Let $\calF=(V_0,\dots,V_n)$ be a complete flag of a vector space $V$, i.e., a maximal increasing list of linear subspaces of $V$.
An endomorphism $u$ of $V$ is said to leave $\calF$ invariant when it leaves $V_i$ invariant for all $i \in \lcro 0,n\rcro$.
We denote by $\calT_\calF$ the space of all endomorphisms of $V$ that leave $\calF$ invariant.
\end{Not}

A basis $(e_1,\dots,e_n)$ is said to be adapted to $\calF$ whenever $V_i=\Vect(e_1,\dots,e_i)$ for all $i \in \lcro 0,n\rcro$.
If we take an adapted basis $\bfB$ for the flag $\calF$,
the space $\calT_\calF$ consists of all the endomorphisms of $V$ whose matrix in $\bfB$ is upper-triangular.
In particular $\dim \calT_\calF=\frac{n(n+1)}{2}\cdot$

\begin{theo}
Let $\F$ be a finite field with odd characteristic.
Let $V$ be an $n$-dimensional vector space, and $\calS$ be a weakly triangularizable linear subspace of $\End(V)$ with dimension $\frac{n(n+1)}{2}\cdot$
Then $\calS=\calT_\calF$ for some complete flag $\calF$ of $V$.
\end{theo}

It is known that for any \emph{optimal} weakly triangularizable linear subspace $\calS$ of $\End(V)$,
the $\calS$-invariant subspaces of $V$ are totally ordered by inclusion (see the appendix of \cite{dSPtriang}, in which this property is shown to hold for a large
class of similar problems). Hence if we have a complete flag of invariant subspaces for such a space $\calS$, then this flag contains \emph{all} the $\calS$-invariant subspaces,
and in particular it is uniquely determined by $\calS$.

To find a complete flag of invariant subspaces, which is equivalent to finding an adapted basis of this flag,
we will start from an $\calS$-adapted vector $x$, and will take $e_n:=x$ as obviously $x$ is a very good candidate for the last vector.
Then we will consider the subspace of all $u \in \calS$ that map $e_n$ into $\F e_n$, and consider the induced endomorphisms
of $V/\F e_n$. It will be seen that the resulting space is weakly triangularizable, and optimal, and we will find a corresponding flag
$\overline{\calF}$ of $V/\F e_n$. Moreover, it will be seen that $\calS$ contains a unique idempotent $\pi$ with range $\F e_n$.
Then, we will lift an adapted basis of $\overline{\calF}$ to a basis $(e_1,\dots,e_{n-1})$ of $\Ker \pi$,
thereby obtaining a basis $(e_1,\dots,e_n)$ of $V$.
Of course, at this point we will want to prove that the flag $\calF:=(\Vect(e_1,\dots,e_i))_{0 \leq i \leq n}$
is invariant under $\calS$, but there does not seem to be a quick proof of this, and a much deeper analysis will be required.

Using a classical approach, we will push the previous idea further by taking $H:=\Vect(e_2,\dots,e_n)$ and by considering the subspace
$\calS''$ consisting of the $u \in \calS$ that leave $H$ invariant. Then we will prove that $H$ is $\calS$-adapted and that
$\{u_H \mid u \in \calS''\}$ is an optimal weakly triangularizable linear subspace of $\End(H)$. By induction,
we will find a corresponding flag, and we will then prove -- this is called the diagonal compatibility step -- that this flag is exactly
$(\Vect(e_2,\dots,e_i))_{1 \leq i \leq n}$.

The remainder of the proof will consist in exploiting all the previous results in matrix form: we will obtain
very simple matrices that represent specific elements of $\calS$ in the basis $(e_1,\dots,e_n)$, and by linearly combining these
matrices we will use the triangularizability assumption to simplify their form, until we find that every upper-triangular matrix corresponds to an element of $\calS$
in the basis $(e_1,\dots,e_n)$.

The proof is organized as follows. In Section \ref{section:n=2}, we settle the case $n=2$, which cannot be done by induction.
In Section \ref{section:lemmapol}, we state and prove a lemma on polynomials over finite fields, which is used in late steps of our proof of Theorem \ref{theo:main}.
The inductive part of Theorem \ref{theo:main} is performed in Section \ref{section:induction}.

\vskip 3mm
\begin{center}
\emph{From now on, we systematically assume that $\F$ is finite with odd characteristic.}
\end{center}

\section{The case $n=2$}\label{section:n=2}

Let $V$ be a $2$-dimensional vector space, and $\calS$ be an optimal weakly triangularizable subspace of $\End(V)$.
Since $\F \id_V+\calS$ is weakly triangularizable, the optimality of $\calS$ yields $\id_V \in \calS$.

Let us consider the orthogonal complement $\calS^\bot$ of $\calS$ for the non-degenerate symmetric bilinear form
$(u,v) \mapsto \tr(uv)$ on $\End(V)$. Since $\dim \calS=3$ we have $\calS^\bot=\F v_0$ for some $v_0 \in \End(V) \setminus \{0\}$,
and since $\id_V \in \calS$ we have $\tr (v_0)=0$.

In particular, because $\car(\F) \neq 2$ the endomorphism $v_0$ is not a scalar multiple of the identity, and hence we can find
a vector $\overrightarrow{\jmath} \in V$ such that $(\overrightarrow{\jmath},v_0(\overrightarrow{\jmath}))$ is a basis of $V$.
Then, in the basis $\bfB=(v_0(\overrightarrow{\jmath}),\overrightarrow{\jmath})$ the matrix of $v_0$ equals
$\begin{bmatrix}
0 & 1 \\
\beta & 0
\end{bmatrix}$ for some $\beta \in \F$, and we deduce that $\calS$ is represented in $\bfB$ by the matrix space
$$\left\{\begin{bmatrix}
x & z \\
-\beta z & y
\end{bmatrix} \mid (x,y,z) \in \F^3 \right\}.$$
It follows that for all $(x,y,z) \in \F^3$ the polynomial $t^2-(x+y)t+(xy+\beta z^2)$ splits over $\F$.
Taking $y=-x$ yields that $x^2-\beta z^2$ is a square in $\F$ for all $(x,z) \in \F^2$.
Yet, this is known to fail if $\beta\neq 0$: indeed in that case $(x,z) \in \F^2 \mapsto x^2-\beta z^2$
is a regular quadratic form of rank $2$, and since $\F$ is finite this form must take all possible values in $\F$
(this is also a special case of the Chevalley-Warning theorem).
Yet since $\car(\F) \neq 2$ it is known that not all elements of $\F$ are squares (the mapping $x \in \F^\times \mapsto x^2 \in \F^\times$
is non-injective because it is a group homomorphism and its kernel contains $-1_\F$).
Therefore $\beta=0$ and we conclude that
$\calS$ is represented by $\MatT_2(\F)$ in the basis $(v_0(\overrightarrow{\jmath}),\overrightarrow{\jmath})$. This completes our proof.

\begin{Rem}
A close look at the proof of Theorem \ref{theo:classNRC} featured in \cite{dSPtriang} reveals that it
shows that every weakly triangularizable subspace of $\Mat_2(\F)$ with dimension $3$
is conjugated either to $\Matt_2(\F)$ or to $\Mats_2(\F)$, provided that $\F$ has characteristic other than $2$ and is not quadratically closed.
In our setting, where $\F$ is finite with odd characteristic, the second option is then ruled out by the fact that $\F$ is not Pythagorean.
However, this case is so elementary that we preferred to give a direct proof.
\end{Rem}

\section{A lemma on split polynomials over finite fields}\label{section:lemmapol}

This short section is devoted to the proof of the following lemma.

\begin{lemma}\label{lemma:pol}
Let $\F$ be a finite field with $|\F|>2$.
Let $p,q$ be monic polynomials of respective degrees $d$ and $d-1$ over $\F$.
Assume that $p-\lambda q$ splits over $\F$ for all $\lambda \in \F$. Then $q$ divides $p$.
\end{lemma}

If $|\F|=2$, the result fails whenever $d$ is odd, with the counterexample $p:=t^d$ and $q:=t^d-(t-1)^d$.

\begin{proof}
Let us consider the (monic) greatest common divisor $r$ of $p$ and $q$. Write $p =rp_1$ and $q=r q_1$ for polynomials $p_1$ and $q_1$ with respective degrees
$e$ and $e-1$ for some integer $e \geq 1$. Then for all $\lambda \in \F$, since $p-\lambda q=r(p_1-\lambda q_1)$,
we get that $p_1-\lambda q_1$ splits over $\F$. This reduces the situation to the case where $p$ and $q$ are relatively prime, in which case
we need to prove that $d=1$ (so that $q=1$ and hence $q \mid p$).

So, assume that $p$ and $q$ are relatively prime and that $d \geq 2$. Set $h_\lambda:=p-\lambda q$ for $\lambda \in \F$.
For distinct $\lambda,\mu$ in $\F$, the polynomials
$h_\lambda$ and $h_\mu$ have no common root, since any common root of them would be a common root of $p$ and $q$.
Yet $h_\lambda$ has at least one root in $\F$ for all $\lambda \in \F$ (because $\deg(h_\lambda)=d \geq 1$).
Since $\F$ is finite it follows that each $h_\lambda$ is split with exactly one root in $\F$, which we denote by $z_\lambda$.
By Vieta's formula applied to $h_\lambda$, we deduce that each mapping
$$\lambda \mapsto d z_\lambda, \quad \lambda \mapsto d (z_\lambda)^{d-1} \quad \text{and} \quad \lambda \mapsto (z_\lambda)^d$$
is affine on $\F$. More specifically the first one equals $\lambda \mapsto \lambda+\tr p$ and is therefore nonconstant. It follows that $d.1_\F$
is invertible in $\F$, and hence $f : \lambda \mapsto (z_\lambda)^{d-1}$ and $g : \lambda \mapsto (z_\lambda)^d$ are affine mappings,
and $z : \lambda \mapsto z_\lambda$ is an affine mapping with degree $1$.
Since $z : \F \rightarrow  \F$ is surjective and $d \geq 2$, neither $f$ nor $g$ is constant (as both take the values $0$ and $1$) whence both have degree $1$.
With $g=f z$, and noting that $|\F|>2$, we obtain a contradiction by considering degrees of polynomials.
\end{proof}

\section{The inductive step}\label{section:induction}

Remember that $\F$ is a finite field of odd characteristic.
Throughout this section, we let $n \geq 3$ be an integer, and we assume that for every vector space $W$ with dimension $n-1$ over $\F$,
every optimal weakly triangularizable linear subspace of $\End(W)$ equals $\calT_\calF$ for some complete flag $\calF$ of linear subspaces of $W$.

We let $V$ be a vector space with dimension $n$ over $\F$, and we let $\calS \subseteq \End(V)$ be an optimal weakly triangularizable linear subspace.

\subsection{Finding an appropriate basis}\label{section:induction:start}

Our starting point is Proposition \ref{prop:adapted}. It yields an adapted vector $x$ for $\calS$.
We set
$$\calS':=\{u \in \calS : u(x)\in \F x\},$$
which can be viewed as the kernel of the linear mapping $u \in \calS \mapsto u(x) \in V/\F x$.
Hence $\dim \calS \leq (n-1)+\dim \calS'$.

For $u \in \calS'$ we denote by $\overline{u}$ the endomorphism of $V/\F x$
induced by $u$, and set
$$\calT:=\{\overline{u} \mid u \in \calS'\}\subseteq \End(V/\F x).$$

Then $\calT$ is a weakly triangularizable linear subspace of $\End(V/\F x)$, and hence $\dim \calT \leq t_{n-1}(\F)$.
The kernel of $u \in \calS' \mapsto \overline{u} \in \calT$
is $\calS \cap \Hom(V,\F x)$, and since $x$ is $\calS$-adapted the trace yields an injective linear form
$$u \in \calS \cap \Hom(V,\F x) \longmapsto \tr(u) \in \F.$$
As a consequence, $\dim(\calS \cap \Hom(V,\F x)) \leq 1$, with equality attained if and only if the previous linear form is bijective.
By the rank theorem, we deduce that
$$\dim \calS \leq (n-1)+1+\dim \calT \leq n+\dbinom{n}{2}=\dbinom{n+1}{2}=t_{n+1}(\F).$$
Since $\calS$ is optimal, all the previous inequalities turn out to be equalities, with the following consequences:
\begin{itemize}
\item[(i)] $\calT$ is an optimal weakly triangularizable linear subspace of $\End(V/\F x)$;
\item[(ii)] $\calS \cap \Hom(V,\F x)$ has dimension $1$;
\item[(iii)] $\calS \cap \Hom(V,\F x)$ contains a unique element with trace $1$, i.e., a unique rank $1$ idempotent $\pi$;
\item[(iv)] One has $\calS x+\F x=V$, and in fact $\calS x=V$ because $\pi(x)=x$.
\end{itemize}
Now, by induction there is a complete flag $\overline{\calF}$ of $V/\F x$ such that $\calT=\calT_{\overline{\calF}}$.
We can then take an adapted basis  $(f_1,\dots,f_{n-1})$ of $V/\F x$ for this flag, and consider the unique lifting $(e_1,\dots,e_{n-1})$ of it as a basis of $\Ker \pi$.
Then we take $e_n:=x$ and obtain a basis $(e_1,\dots,e_n)$ of $V$.

It will be proved that the flag $\calF:=(\Vect(e_1,\dots,e_k))_{0 \leq k \leq n}$ satisfies $\calT_{\calF} \subseteq \calS$,
but we are still very far from this conclusion.

\subsection{Second call to the induction hypothesis}\label{section:induction:secondstep}

Here, we consider the linear hyperplane $H:=\Vect(e_2,\dots,e_n)$.

\begin{claim}
The hyperplane $H$ is $\calS$-adapted.
\end{claim}

\begin{proof}
Let $u \in \calS$ satisfy $u_{|H}=0$ and $\tr(u)=0$. In particular $u(e_n)=0$, whence $u \in \calS'$, and
the induced endomorphism $\overline{u}$ of $V/\F e_n$ vanishes at all the vectors $f_2,\dots,f_n$, and
$\tr(\overline{u})=\tr(u)=0$.
Since $\calT$ has associated flag $\overline{\calF}$, we derive that $\overline{u}=0$.
Hence $\im u \subseteq \F e_n$, and we conclude that $u=0$ because $e_n$ is $\calS$-adapted.
\end{proof}

Now, we use an argument that is the dual analogue of the one of the previous part of the proof.
We set
$$\calS'':=\{u \in \End(V) : u(H) \subseteq H\}.$$
Every element of $\calS''$ induces an endomorphism $u_H$ of $H$, and we recover
a weakly triangularizable subspace
$$\calW:=\{u_H \mid u \in \calS''\} \subseteq \End(H).$$
The kernel of $u \in \calS'' \mapsto u_H$ consists of the $u \in \calS$ that vanish on $H$,
and because $H$ is $\calS$-adapted the linear form
$$v \in \{u \in \calS : H \subseteq \Ker u\} \longmapsto \tr(v) \in \F$$
is injective on it, leading to $\dim \calS'' \leq 1+\dim \calW$.
Finally, we observe that $\calS''$ is the kernel of the mapping $u \in \calS \mapsto \pi^{(1)} \circ u_{|H} \in \Hom(H,V/H)$,
where $\pi^{(1)} : V \twoheadrightarrow V/H$ denotes the standard projection.
Hence the rank theorem yields
$$\dim \calS \leq (n-1)+\dim \calS'' \leq
(n-1)+1+\dim \calW \leq n+\dbinom{n}{2}=\dbinom{n+1}{2}.$$
Just like in Section \ref{section:induction:start}, we deduce from the equality $\dim \calS=\dbinom{n+1}{2}$ that:
\begin{itemize}
\item[(v)] $\calW$ is an optimal weakly triangularizable subspace of $\End(H)$;
\item[(vi)] $\calS$ contains a (unique) idempotent $\pi'$ with kernel $H$.
\end{itemize}
In particular $\pi' \in \calS'$.

\subsection{The matrix viewpoint}

Now we start representing the elements of $\calS$ by matrices.
We fix the basis $(e_1,\dots,e_n)$ once and for all, and we consider the matrix space $\calM$
that is associated with $\calS$ in this basis.

Throughout, we will need the notation for matrix units: $E_{i,j}$ will denote the matrix of $\Mat_n(\F)$
with exactly one nonzero entry, located at the $(i,j)$-spot and with value $1$, and we write it as $E_{i,j}^{(n)}$
in case of a possible confusion on the size of the matrices.

Now, we translate several of the previous properties in purely matrix terms.

\begin{itemize}
\item[(A1)] The space $\calM$ contains the unit matrix $E_{n,n}$ (which corresponds to $\pi$).
\item[(A2)] For all $U \in \MatT_{n-1}(\F)$, the space $\calM$ contains a unique matrix of the form
$$\begin{bmatrix}
U & [0]_{(n-1) \times 1} \\
[?]_{1 \times (n-1)} & 0
\end{bmatrix},$$
and every matrix in $\calM$ with last column zero is of this form.
This translates the information on the structure of $\calT$, combining it with (A1) and with the fact that $e_n$ is $\calS$-adapted.

\item[(A3)] For every column $C \in \F^n$, the space $\calM$ contains a matrix with last column $C$.
This reformulates point (iv) in Section \ref{section:induction:start}.
\end{itemize}

The matrix that corresponds to the idempotent $\pi'$ has all its columns zero starting from the second one,
so by point (A2) we obtain:

\begin{itemize}
\item[(A4)] The space $\calM$ contains $E_{1,1}+\alpha E_{n,1}$ for some $\alpha \in \F$.
\end{itemize}

\subsection{Diagonal compatibility}

Next, we take a closer look at the structure of $\calW$.
Applying the induction hypothesis to $\calW$, we obtain a complete flag $\calH=(H_0,\dots,H_{n-1})$ of $H$
such that $\calW=\calT_{\calH}$.
We wish to prove that $H_i=\Vect(e_2,\dots,e_{i+1})$ for all $i \in \lcro 0,n-1\rcro$,
and of course to this end it suffices to consider the integers $i \in  \lcro 1,n-2\rcro$.
We denote by $\calM'$ the matrix space that represents $\calW$ in the basis $(e_2,\dots,e_n)$.

To start with, we obtain:
\begin{itemize}
\item[(A5)] For all $U' \in \MatT_{n-2}(\F)$, the space $\calM'$ contains a matrix of the form
$$\begin{bmatrix}
U' & [0]_{(n-2) \times 1} \\
[?]_{1 \times (n-2)} & 0
\end{bmatrix}.$$
This follows directly from (A2).
\item[(A6)] The space $\calM'$ contains $E_{n-1,n-1}^{(n-1)}$. This directly follows from (A1).
\end{itemize}

Let $C' \in \F^{n-1}$. By (A3), some matrix in $\calM$ has last column
$\begin{bmatrix}
0 \\
C'
\end{bmatrix}$. Adding an appropriate matrix given by (A2), it follows that
$\calM$ contains a matrix of the form
$$\begin{bmatrix}
[0]_{1 \times (n-1)} & 0 \\
[?]_{(n-1) \times (n-1)} & C' \\
\end{bmatrix}.$$
Hence, we obtain:
\begin{itemize}
\item[(A7)] For all $X \in \F^{n-1}$, the space $\calM'$ contains a matrix with last column $X$. In other words
$\calW e_n=H$.
\end{itemize}

Now, we can prove the following result:

\begin{claim}
One has $H_i=\Vect(e_2,\dots,e_{i+1})$ for all $i \in \lcro 1,n-2\rcro$.
\end{claim}

\begin{proof}
Our starting point is the observation that $H_{n-2}$ is invariant under $\pi_H$, because obviously $\pi_H \in \calW$
(recall that $\pi_H$ denotes the endomorphism of $H$ induced by $\pi$).
Classically, since $\pi_H$ is diagonalisable its invariant subspaces are sums of subspaces of its eigenspaces, which here
are $\Ker \pi_H$ and $\im \pi_H$. Since $\im \pi_H$ has dimension $1$ and $H_{n-2}$ is a linear hyperplane of $H$, this leaves us with two possibilities only:
Either $H_{n-2}=\Ker \pi_H$ or $\im \pi_H \subseteq H_{n-2}$.

However $\im \pi_H=\Vect(e_n)$, and having $e_n \in H_{n-2}$ would lead to $\calW e_n \subseteq H_{n-2} \subset H$,
contradicting (A7). Hence
$$H_{n-2}=\Ker \pi_H=\Vect(e_2,\dots,e_{n-1}).$$
As a consequence, $\Vect(e_2,\dots,e_{n-1})$ is invariant under $\calW$.
In particular, we can refine (A5) as follows:

\begin{itemize}
\item[(A5')] For all $U' \in \MatT_{n-2}(\F)$, the space $\calM'$ contains the matrix
$$\begin{bmatrix}
U' & [0]_{(n-2) \times 1} \\
[0]_{1 \times (n-2)} & 0
\end{bmatrix}.$$
\end{itemize}
Now we consider the space
$$\calZ:=\{w_{H_{n-2}} \mid w \in \calW\} \subseteq \End(H_{n-2}).$$
Again, it is a weakly triangularizable linear subspace of $\End(H_{n-2})$,
and by (A5') we see that it includes $\calT_{\calF''}$, where
$\calF''=(\Vect(e_2,\dots,e_i))_{1 \leq i \leq n-1}$.
Then $\dim \calZ \leq \dbinom{n-1}{2}=\dim \calT_{\calF''}$
and we deduce that $\calZ=\calT_{\calF''}$. Yet all the subspaces $H_1,\dots,H_{n-2}$
are $\calZ$-invariant, whereas it is known that the $\calT_{\calF''}$-invariant subspaces are precisely the spaces
$\Vect(e_2,\dots,e_i)$ with $i \in \lcro 1,n-1\rcro$.
This proves the claimed statement.
\end{proof}

As a consequence, we now have the following results on the matrix space $\calM$:
\begin{itemize}
\item[(A8)] For all $U \in \Matt_{n-1}(\F)$, the space $\calM$ contains a unique matrix of the form:
$$\begin{bmatrix}
0 & [0]_{1 \times (n-1)} \\
[?]_{(n-1) \times 1} & U
\end{bmatrix}.$$
\item[(A9)] Every matrix $M$ of $\calM$ with first row zero has the following form:
$$M=\begin{bmatrix}
0 & [0]_{1 \times (n-2)} & 0 \\
[?]_{(n-2) \times 1} & [?]_{(n-2) \times (n-2)} & [?]_{(n-2) \times 1} \\
? & [0]_{1 \times (n-2)} & ?
\end{bmatrix}.$$
\end{itemize}

\subsection{Simple matrices in $\calM$}\label{section:simplematrices}

From now on, all the matrices of $\calM$ will be represented according to the following format:
$$\begin{bmatrix}
? & [?]_{1 \times (n-2)} & ? \\
[?]_{(n-2) \times 1} & [?]_{(n-2) \times (n-2)} & [?]_{(n-2) \times 1} \\
? & [?]_{1 \times (n-2)} & ?
\end{bmatrix},$$
so we will avoid systematically recalling the size of the blocks and simply write $0$ for a zero matrix, whatever its size.

Next, we apply (A2) to obtain that for all $L \in \Mat_{1,n-2}(\F)$ the space
$\calM$ contains a unique matrix of the form
$$A_L=\begin{bmatrix}
0 & L & 0 \\
0 & 0 & 0 \\
f(L) & \varphi(L) & 0
\end{bmatrix},$$
and $f : \Mat_{1,n-2}(\F) \rightarrow \F$ and $\varphi : \Mat_{1,n-2}(\F) \rightarrow \Mat_{1,n-2}(\F)$
are linear mappings.

Applying (A8) yields that, for all $C \in \F^{n-2}$, the space
$\calM$ contains a unique matrix of the form
$$B_C=\begin{bmatrix}
0 & 0 & 0 \\
\psi(C) & 0 & C \\
g(C) & 0 & 0
\end{bmatrix},$$
and $g : \F^{n-2} \rightarrow \F$ and $\psi : \F^{n-2} \rightarrow \F^{n-2}$
are linear mappings.

Let finally $U \in \MatT_{n-2}(\F)$. We know from point (A2) that $\calM$ contains a matrix of the form
$$\begin{bmatrix}
0 & 0 & 0 \\
0 & U & 0 \\
? & ? & 0
\end{bmatrix}.$$
Combining this with point (A9) yields that such a matrix is actually of the form
$$\begin{bmatrix}
0 & 0 & 0 \\
0 & U & 0 \\
? & 0 & 0
\end{bmatrix}.$$
Finally, thanks to the uniqueness in point (A2) we conclude that, for all $U \in \MatT_{n-2}(\F)$, the space
$\calM$ contains a unique matrix of the form
$$G_U=\begin{bmatrix}
0 & 0 & 0 \\
0 & U & 0 \\
h(U) & 0 & 0
\end{bmatrix},$$
and $h : \MatT_{n-2}(\F) \rightarrow \F$ is a linear form.

Our aim is to prove that all the mappings $f,g,h,\varphi,\psi$ vanish, and from there the conclusion will be quite close.
We will actually start by proving that $f,g$ and $\varphi$ vanish.

\subsection{Analyzing $f,g,\varphi$}\label{section:vanishfgphi}

As is customary in similar proofs \cite{dSPsoleeigenvalue,dSPlargeaffinenonsingular,dSPGerstenhaberskew,dSPfeweigenvalues},
we will combine the $A_L$ and $B_C$ matrices to retrieve as much information as we can from the assumption that
$\calM$ is weakly triangularizable.
So, let us fix $C \in \F^{n-2} \setminus \{0\}$ and $L \in \Mat_{1,n-2}(\F)$.

We take an arbitrary $\lambda \in \F$.
Let us consider the matrix $A_L+B_C+\lambda E_{n,n}$, and the endomorphism $u \in \calS$ that corresponds to it.
We consider the vector $y:=\underset{k=2}{\overset{n-1}{\sum}} c_k\,e_k$ and we note that $u(e_n)=y+\lambda e_n$.
Next,
$$u(y)=LC\,e_1+\varphi(L)C\, e_n,$$
$$u(e_1)=z+(f(L)+g(C))\,e_n \quad \text{where} \quad z:=\sum_{k=2}^{n-1} \psi(C)_k\, e_k,$$
and finally
$$u(z) =L\psi(C)\,e_1+\varphi(L) \psi(C)\,e_n.$$
Now, we split the discussion into three cases.

\vskip 2mm
\noindent
\textbf{Case 1: $y,z$ are linearly independent and $LC \neq 0$.}

Then $e_n,y,e_1,z$ are linearly independent, the subspace they span is invariant under $u$, and the resulting matrix of the induced endomorphism in the basis
$(e_n,y,e_1,z)$ equals
$$\begin{bmatrix}
\lambda & \varphi(L)C & f(L)+g(C) & \varphi(L)\psi(C) \\
1 & 0 & 0 & 0 \\
0 & LC & 0 & L\psi(C) \\
0 & 0 & 1 & 0 \\
\end{bmatrix}.$$
Hence, the characteristic polynomial of the latter matrix splits over $\F$.
By expanding with respect to the first column, we rewrite this characteristic polynomial as
$p-\lambda q$, where $p$ stands for the characteristic polynomial of $$\begin{bmatrix}
0 & \varphi(L)C & f(L)+g(C) & \varphi(L)\psi(C) \\
1 & 0 & 0 & 0 \\
0 & LC & 0 & L\psi(C) \\
0 & 0 & 1 & 0
\end{bmatrix},$$
and $q$ stands for the one of
$\begin{bmatrix}
0 & 0 & 0 \\
LC & 0 & Lg(C) \\
0 & 1 & 0
\end{bmatrix}$.
Hence $p$ and $q$ have respective degrees $4$ and $3$, and the assumptions of Lemma \ref{lemma:pol}
are satisfied (indeed, $u$ is triangularizable whatever the choice of $\lambda$). Therefore $q$ divides $p$.

Yet, by expanding with respect to the first column, we find $p=tq+r$, where
$$r:=\begin{vmatrix}
-\varphi(L)C & -f(L)-g(C) & -\varphi(L)\psi(C) \\
-LC & t & -L\psi(C) \\
0 & -1 & t
\end{vmatrix}.$$
It follows that $q$ divides $r$.
Yet clearly $\deg(r) \leq 2$, so we deduce that $r=0$.
In particular, the coefficient $-\varphi(L)C$ of $r$ on $t^2$ vanishes. Then
we expand with respect to the first column and find $LC(-(f(L)+g(C))t-\varphi(L) \psi(C))=0$,
and since $LC \neq 0$ we conclude that $f(L)+g(C)=0$.

\vskip 2mm
\noindent
\textbf{Case 2: $y,z$ are linearly dependent and $LC \neq 0$.}

We write $z=\beta y$ for some $\beta \in \F$.
Then $e_n,y,e_1$ are linearly independent, the subspace they span is invariant under $u$, and the resulting matrix of the induced endomorphism in the basis
$(e_n,y,e_1)$ equals
$$\begin{bmatrix}
\lambda & \varphi(L)C & f(L)+g(C) \\
1 & 0 & \beta \\
0 & LC & 0
\end{bmatrix}.$$
Then, with exactly the same line of reasoning as in Case 1, we use Lemma \ref{lemma:pol} to obtain
$$\begin{vmatrix}
-\varphi(L)C & -f(L)-g(C) \\
-LC & t
\end{vmatrix}=0,$$
and once more we use $LC \neq 0$ to deduce that $\varphi(L)C=0$ and $f(L)+g(C)=0$.

\vskip 2mm
\noindent
\textbf{Case 3: $LC=0$.}

Here $\Vect(e_n,y)$ is invariant under $u$ and the resulting endomorphism
has matrix $\begin{bmatrix}
\lambda & \varphi(L)C \\
1 & 0
\end{bmatrix}$ in the basis $(e_n,y)$. Applying Lemma \ref{lemma:pol} once more yields $\varphi(L)C=0$.

In particular, we have shown in any case that $\varphi(L)C=0$ for all $L \in \Mat_{1,n-2}(\F)$ and all $C \in \F^{n-2} \setminus \{0\}$.
Hence:

\begin{claim}
The mapping $\varphi$ vanishes.
\end{claim}

Here is the next step.

\begin{claim}
The mappings $f$ and $g$ vanish.
\end{claim}

\begin{proof}
We have shown through Cases 1 and 2 that $f(L)+g(C)=0$ for all $(L,C)\in \Mat_{1,n-2}(\F) \times \F^{n-2}$ such that $LC \neq 0$.

Fix $C_0 \in \F^{n-2} \setminus \{0\}$. Then the mapping $L \mapsto (f(L)+g(C_0))LC_0$ vanishes on $\Mat_{1,n-2}(\F)$.
Yet, it is the product of two polynomial mappings of degree at most $1$, including $L \mapsto LC_0$, which is nonzero.
Since $|\F|>2$ we deduce that $L \mapsto f(L)+g(C_0)$ vanishes on $\Mat_{1,n-2}(\F)$,
which yields $f=0$ and $g(C_0)=0$. Varying $C_0$ then yields $g=0$.
\end{proof}

At this point, we cannot obtain the vanishing of $\psi$ in a straightforward way. Instead, we will turn to the
idempotent $\pi'$ and complete the analysis of its structure.

\subsection{The idempotent $\pi'$}

The idempotent $\pi'$ is represented by $E_{1,1}+\alpha E_{n,1}$ in the basis $(e_1,\dots,e_n)$ (see property (A4)).
We will use a perturbation trick to prove that $\alpha=0$.

\begin{claim}
One has $\alpha=0$.
\end{claim}

\begin{proof}
We use the observation that the choice of the vectors $e_1,\dots,e_{n-1}$ could have been slightly different,
and in particular we could have used $e_2+e_1$ instead of $e_2$. With this choice, the way we found $\pi'$
can be applied to yield an idempotent $\pi'' \in \calS$ with kernel $\Vect(e_2+e_1,e_3,\dots,e_{n-1},e_n)$ and range $\Vect(e_1+\beta e_n)$
for some $\beta \in \F$. Hence the matrix of $\pi''$ in $(e_1,\dots,e_n)$ equals $E_{1,1}-E_{1,2}+\beta E_{n,1}-\beta E_{n,2}$.
Yet, as now we have $f=0$ and $\varphi=0$, we know that $\calM$ contains $E_{1,2}$, and by adding it and subtracting $E_{1,1}+\alpha E_{n,1}$
we deduce that $\calM$ contains $-\beta E_{n,2}+(\beta-\alpha) E_{n,1}$. Since $e_n$ is $\calS$-adapted this yields
$-\beta=0$ and $\beta-\alpha=0$, and we conclude that $\alpha=0$.
\end{proof}

As a consequence, (A4) is now refined as:
\begin{itemize}
\item[(A4')] The space $\calM$ contains $E_{1,1}$.
\end{itemize}

\subsection{The vanishing of $\psi$}

Now, we can finally obtain that $\psi=0$.

\begin{claim}
The mapping $\psi$ vanishes.
\end{claim}

\begin{proof}
Here we will use a symmetry observation.
We consider the permutation matrices $P:=(1_{i+j=n+1})_{1 \leq i,j \leq n} \in \GL_n(\F)$
and $Q:=(1_{i+j=n-1})_{1 \leq i,j \leq n-2} \in \GL_{n-2}(\F)$, and the matrix space $\calM^\star:=P \calM^T P^{-1}$.

By (A4'), we see that $\calM^\star$ contains $E_{n,n}$.
We observe thanks to the vanishing of $\varphi$ and $f$ that, for all $C \in \F^{n-2}$, the space $\calM^\star$ contains
$$B'_C:=\begin{bmatrix}
0 & 0 & 0 \\
0 & 0 & C \\
0 & 0 & 0
\end{bmatrix},$$
whereas from the vanishing of $g$ we see that, for all $L \in \Mat_{1,n-2}(\F)$, the space $\calM^\star$ also contains
$$A'_L:=\begin{bmatrix}
0 & L & 0 \\
0 & 0 & 0 \\
0 & \varphi'(L) & 0
\end{bmatrix} \quad \text{where $\varphi'(L)=(Q\psi(Q L^T))^T$.}$$
Since $\calM^\star$ is weakly triangularizable and has dimension $\frac{n(n+1)}{2}$, we can then directly follow the
line of reasoning of Section \ref{section:vanishfgphi} to find
that $\varphi'$ vanishes. This yields the claimed statement.
\end{proof}

\subsection{Completing the proof}

\begin{claim}
The space $\calM$ contains the unit matrix $E_{1,n}$.
\end{claim}

\begin{proof}
We will use another perturbation argument.

As a consequence of the identities $g=0$ and $\psi=0$, along with the presence of $\pi$ in $\calS$, we have found that $\calS$ contains every endomorphism of $V$
that vanishes on $\Ker \pi$ and maps $e_n$ into $\Vect(e_2,\dots,e_n)$.
Yet, once more there is some slack in the choice of $(e_2,\dots,e_{n-1})$, and in particular we can
harmlessly replace $e_2$ with $e_2+e_1$ without changing the other vectors.
Then the previous proof yields that $\calS$ also contains all the endomorphisms of $V$
that vanish on $\Ker \pi$ and map $e_n$ into $\Vect(e_2+e_1,e_3,\dots,e_n)$.

By summing, we deduce that $\calS$ contains all the endomorphisms of $V$ that vanish on $\Ker \pi$,
and in particular $\calM$ contains $E_{1,n}$.
\end{proof}

\begin{claim}
One has $h(U)=0$ for all $U \in \MatT_{n-2}(\F)$.
\end{claim}

\begin{proof}
Let $U \in \MatT_{n-2}(\F)$. Recall the notation $G_U$ from Section \ref{section:simplematrices}.

Let $\lambda \in \F$. Denote by $u$ the endomorphism of $V$ with matrix $G_U+ \lambda E_{1,n}$ in the basis $(e_1,\dots,e_n)$.
Then $u \in \calS$ and hence $u$ is triangularizable. We observe that $u$ leaves $\Vect(e_1,e_n)$ invariant and
the matrix of the induced endomorphism in the basis $(e_1,e_n)$ is $\begin{bmatrix}
0 & \lambda \\
h(U) & 0
\end{bmatrix}$. Hence the latter is triangularizable, and since its characteristic polynomial is $t^2-\lambda h(U)$
we deduce that $\lambda h(U)$ is a square in $\F$. If $h(U) \neq 0$, varying $\lambda$ would yield that all the elements of $\F$
are squares, which is known to be false because $\F$ is finite with odd characteristic. Hence $h(U)=0$.
\end{proof}

Now we can conclude. We have shown that $\calM$ contains $E_{i,j}$ for all $i,j$ in $\lcro 1,n\rcro$ such that $i \leq j$.
Hence $\MatT_n(\F) \subseteq \calM$. Since the dimensions are equal, we deduce that $\calM=\MatT_n(\F)$, which means that
$\calS=\calT_\calF$ for the flag $\calF=(\Vect(e_1,\dots,e_i))_{0 \leq i \leq n}$.
This finishes the inductive step, thereby completing the proof of Theorem \ref{theo:main}.

\end{document}